\documentclass{amsart}
\usepackage{amsfonts,amssymb,amsmath,amsthm}
\usepackage{url}
\usepackage{enumerate}

\newtheorem{thrm}{Theorem}[section]
\newtheorem{lem}[thrm]{Lemma}

\newtheorem{cor}[thrm]{Corollary}
\theoremstyle{definition}
\newtheorem{definition}[thrm]{Definition}
\newtheorem{remark}[thrm]{Remark}
\numberwithin{equation}{section}

\author{Hiroshi~Nozaki}
\address{
Graduate School of Information Sciences, \\
	Tohoku University \\
	Aoba-ku, Sendai 980-8579, \\
	Japan}
\email{nozaki@ims.is.tohoku.ac.jp}
\author{Masanori~Sawa}
\address{
Graduate School of Information Sciences, \\
	Nagoya University \\
	Chikusa-ku, Nagoya 464-8601, \\
	Japan}
\email{sawa@is.nagoya-u.ac.jp}

\thanks{This research started while
the authors were visiting scholars at the University of Texas at Brownsville.
This research is supported by
the Japan Society for the Promotion of Science.
}

\begin{document}

\title[Note on cubature formulae and
designs obtained from group orbits]{Note on cubature formulae and
designs obtained from group orbits}

\begin{abstract}
In 1960,
Sobolev proved that for a finite reflection group $G$,
a $G$-invariant cubature formula is of degree $t$ if and only if
it is exact for all $G$-invariant polynomials of degree at most $t$.
In this paper,
we find some observations on invariant cubature formulas and Euclidean designs
in connection with the Sobolev theorem.
First, we give an alternative proof of
theorems by Xu (1998) on necessary and sufficient conditions
for the existence of cubature formulas with some strong symmetry.
The new proof is shorter and simpler compared to the original one by Xu, and
moreover gives a general interpretation of
the analytically-written conditions of Xu's theorems.
Second,
we extend a theorem by Neumaier and Seidel (1988) on
Euclidean designs to invariant Euclidean designs, and thereby
classify tight Euclidean designs obtained from
unions of the orbits of the corner vectors. 
This result generalizes a theorem of Bajnok (2007) which classifies
tight Euclidean designs invariant under the Weyl group of type $B$
to other finite reflection groups.
\end{abstract}
\maketitle

\section{Introduction} \label{sect1}

A main problem of numerical integration is to approximate the integral
\begin{equation*}
\int_\Omega f(x) d \mu.
\end{equation*}
Here $x$ is an $n$-dimensional coordinate vector and
$\mu$ is a probability
measure on a domain $\Omega$ in $\mathbb{R}^n$.
We search for an approximation formula by
taking a positive linear combination of
the function values of $f$ at specified points
$x_1,\cdots,x_N$, that is,
\begin{equation}
\label{eq:CF}
\sum_{i = 1}^N w_i f(x_i).
\end{equation}
We call (\ref{eq:CF}) a {\it cubature formula}.
The values $w_i$ are the {\it weights}
and $x_i$ are the {\it points}
of a cubature formula.
To each formula we assign the set of functions
for which it is exact.
Most often this set is the space of all polynomials of degree no more than 
$t$; in this case
a cubature formula is said to be {\it of degree $t$}.
We refer the readers to the comprehensive monograph~\cite{Dunkl-Xu,Stroud}
for the basic theory of cubature formula.

A fundamental objective is to construct
cubature formulas of large degrees with few points.
The requirement that a given cubature formula is exact for polynomials
up to a certain degree can be reduced to the problem of
solving a system of algebraic equations.
In general,
the larger the number of points or the degree of a cubature formula 
is,
the greater the size of this system is.
Sobolev~\cite{Sobolev} gave a celebrated criterion to diminish
the size of the system to be solved.
Namely, he proved that
an invariant cubature formula is of degree $t$
if and only if
it is exact for all polynomials of degree at most $t$
invariant under the group.
This is known as the Sobolev theorem.
The Sobolev theorem is widely accepted by
the cubature community in analysis and related areas;
for instance see~\cite{Mysovskikh,Salikhov}.

Independent of the line of research in analysis and related areas,
Goethals and Seidel~\cite[Theorem 3.12]{Goethals-Seidel}
developed the invariant theory of
Chebyshev-type cubature formulas on the sphere or spherical designs.
As a generalization of spherical designs
Neumaier and Seidel~\cite{Neumaier-Seidel} considered
cubature formulas on several concentric spheres
called Euclidean designs.
Bajnok~\cite{Bajnok} classified
tight Euclidean designs whose points are
the union of the orbits of the corner vectors
of the group $B_n$, and in particular, he obtained
several new tight designs.
Here a Euclidean design is tight if
it is minimal with respect to
a lower bound for the number of points.
To obtain the results,
Bajnok~\cite[Proposition 14]{Bajnok} essentially used the idea of the
Sobolev theorem for $B_n$-invariant Euclidean
designs, though he did not offer the name of Sobolev.
It seems that
some researchers in combinatorics and related areas do not fully
recognized the Sobolev theorem~\cite{Eiichi}.

In this paper we find some observations on invariant cubature formulas
in connection with the Sobolev theorem.
In Section 2 we explain Sobolev's invariant theory in detail.
We also explain some basic facts related to Euclidean designs, e.g.,
a theorem of Neumaier and Seidel~\cite{Neumaier-Seidel} which
is well known in algebra and combinatorics.
In Section 3 we give an alternative proof of
famous theorems by Xu~\cite[Theorem 1.1, Theorem 1.2]{Xu}
on necessary and sufficient conditions
for the existence of cubature formulas with radial symmetry.
The original proof by Xu requires some tedious calculations and
technical tools in numerical analysis like,
Gaussian-Lobatto quadrature, Gaussian-Radau quadrature.
Eventually it is long, and
researchers in other areas may not be familiar with his proof.
Whereas,
our new proof
is short and simple compared to the original proof.
Moreover it gives a general interpretation
of the analytically-written conditions of Xu's theorems, and so
will be readable and acceptable for researchers not only in analysis,
but also in other areas like algebra and combinatorics.
In Section 4
we extend the theorem of Neumaier and Seidel
to invariant Euclidean designs, and thereby
classify tight Euclidean designs obtained from
unions of the orbits of the corner vectors.
This classification generalizes the result of Bajnok
for other finite reflection groups.

\section{Preliminaries}
\label{sect2}

Let ${\rm Hom}_l(\mathbb{R}^n)$ be the linear space of
all real homogeneous polynomials of total degree $l$ in $n$ variables.
Let $\mathcal{P}_l(\mathbb{R}^n) = \sum_{i=0}^l
\text{Hom}_i(\mathbb{R}^n)$,
$\mathcal{P}_l^*(\mathbb{R}^n)
= \sum_{i=0}^{\lfloor l/2 \rfloor} {\rm Hom}_{l-2i}(\mathbb{R}^n)$.
We denote by
$\text{Harm}_l(\mathbb{R}^n)$
the subspace of $\mathcal{P}_l(\mathbb{R}^n)$ of
harmonic homogeneous polynomials of degree $l$.
Let $\mathcal{P}_l(A), \mathcal{P}_l^*(A)$ be the space of functions which
are the restrictions of the corresponding polynomials to
$A \subset \mathbb{R}^n$.

Let $G$ be a finite subgroup of the orthogonal group in $\mathbb{R}^n$ and
$f \in \mathcal{P}_l(\mathbb{R}^n)$.
We consider the action of $\sigma \in G$ on $f$ as follows:
\begin{align*}
(\sigma f) (x) = f(x^{\sigma^{-1}}), \qquad x \in \mathbb{R}^n.
\end{align*}
A polynomial $f$ is said to be {\it $G$-invariant} if it satisfies that
\begin{align*}
\sigma f = f,\qquad \forall \sigma \in G.
\end{align*}
We denote by $\mathcal{P}_l(\mathbb{R}^n)^G, {\rm Harm}_l(\mathbb{R}^n)^G$
the set of $G$-invariant polynomials in
$\mathcal{P}_l(\mathbb{R}^n)$,
${\rm Harm}_l(\mathbb{R}^n)$ respectively.

A cubature formula (\ref{eq:CF}) is said to be {\it invariant under $G$}, or
{\it $G$-invariant} if the domain $\Omega$ and measure $\mu$ of the integral
are invariant under $G$ and
the set of points is the union of $G$-orbits and
to each point of the same orbit an equal weight is assigned.
The following is known as the Sobolev theorem.

\begin{thrm}
[\cite{Sobolev}]
\label{thm:Sobolev}
With the above set up,
a $G$-invariant cubature formula is of degree $t$ if and only if
it is exact for every polynomial
$f \in \mathcal{P}_t(\mathbb{R}^n)^G$.
\end{thrm}
\noindent
The Sobolev theorem is widely accepted by
the cubature community in analysis and related areas:
In particular Russian mathematicians in analysis have developed
the Sobolev theorem and employed it to construct many cubature formulas;
for instance see~\cite{Mysovskikh,Salikhov}.
Xu~\cite{Xu} presented beautiful theorems on the existence and structure
of cubature formulas for radially symmetric integrals.
In section 3 we review his theorems in detail and
give a short proof using the Sobolev theorem.

Next let us explain a combinatorial object
called Euclidean design which was
introduced by Neumaier and Seidel~\cite{Neumaier-Seidel}.
Let $X$ be a finite set in $\mathbb{R}^n$.
Let $r_1, r_2, \cdots, r_p$ be the norms of the vectors in $X$.
For $i= 1, 2, \cdots, p$
we denote by $S_i^{n-1}$ the sphere of radius $r_i$ centered at the origin, 
namely, $S_i^{n-1} = \{ x \in \mathbb{R}^n \mid \| x \| = r_i \}$, and
let $X_i = X \cap S_i^{n-1}$.
The set $X$ is said to be
{\it supported by $p$ concentric spheres $S=\bigcup_{i=1}^p S_i^{n-1}$}.
To each $S_i$ we assign the surface measure $\rho_i$.
Let $|S_i^{n-1}| = \int_{S_i^{n-1}}{\rm d}\rho_i(x)$,
with the convention that
$\frac{1}{|S_i^{n-1}|} \int_{S_i^{n-1}}
f(x) {\rm d} \rho_i (x) = f(0)$ if $S_i^{n-1} = \{ 0 \}$.
\begin{definition}
\label{def:Euclid}
With the same notations as in the above paragraph, we say
$X$ is a {\it Euclidean $t$-design supported by $S$}
if there exists a positive weight function $w(x)$ on $X$ such that
\begin{equation*}
\sum_{i = 1}^p \frac{\sum_{x \in X_i} w(x)}{|S_i^{n-1}|}
\int_{S_i^{n-1}} f(x) {\rm d} \rho_i (x) 
= \sum_{x \in X} w(x) f(x)
\end{equation*} 
for every polynomial $f \in \mathcal{P}_t(S)$. 
\end{definition}
\noindent
We can regard a Euclidean design as
a cubature formula on some concentric spheres.
Conversely a cubature formula for a class of integral with
some symmetry is a Euclidean design
(cf.~\cite[Lemma 3.1]{Hirao-Sawa}).
The following theorem by Neumaier and Seidel is
well known in algebra and combinatorics.
\begin{thrm}
[\cite{Neumaier-Seidel}]
With the same notations as in Definition~\ref{def:Euclid},
the following are equivalent:
\begin{enumerate}
\item[{\rm (i)}] $X$ is a Euclidean $t$-design with a weight function $w$.
\item[{\rm (ii)}] $\sum_{x \in X} w(x) f(x) = 0$ for
every $f \in ||x||^{2j} {\rm Harm}_l(\mathbb{R}^n)$
with $1 \le l \le t, 0 \le j \le \lfloor \frac{t-l}{2} \rfloor$.
\end{enumerate}
\label{thm:NS88}
\end{thrm}
\noindent
In Section 4 we give a stronger theorem than Theorem \ref{thm:NS88} 
for invariant Euclidean designs, especially for
researchers in algebra and combinatorics. 

Define $p'=p-\varepsilon_{S}$, where
$\varepsilon_{S} = 1$ if $0 \in S$, and $\varepsilon_{S} = 0$ otherwise.
The dimensions of $\mathcal{P}_l(S)$ and $\mathcal{P}_l^{\ast}(S)$ are well known.
\begin{thrm}[\cite{
Delsarte-Seidel,Erdelyi}] 
Let $S\subset \mathbb{R}^n$. 
\begin{enumerate}
\item 
$
\dim\mathcal{P}_l(S)=
\left\{\begin{array}{ll}
\varepsilon_{S} + \sum_{i=0}^{2p'-1} \binom{n+l-i-1}{n-1} &
\qquad \text{if $l \geq 2p'$}, \\
\dim \mathcal{P}_l(\mathbb{R}^n)=\binom{n+l}{l} &
\qquad \text{if $l \leq 2p'-1$.} 
\end{array} \right.
$

\item 
$
\dim \mathcal{P}_l^{\ast}(S)=
\left\{\begin{array}{ll}
\varepsilon_{S}+\sum_{i=0}^{p'-1} \binom{n+l-2i-1}{n-1} &
\qquad \text{if $l$ is even, $l\geq 2p'$},\\
\sum_{i=0}^{p'-1} \binom{n+l-2i-1}{n-1} &
\qquad \text{if $l$ is odd, $l\geq 2p'$},\\
\dim \mathcal{P}_l^{\ast}(\mathbb{R}^n)
=\sum_{i=0}^{\lfloor \frac{l}{2} \rfloor} \binom{n+l-2i-1}{n-1} &
\qquad \text{if $l \leq 2p'-1$}.
\end{array} \right.
$
\end{enumerate}
\end{thrm}
The following lower bounds are known as the Fisher-type
inequality for the size of a Euclidean design~\cite{Bannai-Bannai-Hirao-Sawa,
Delsarte-Seidel, Moller2, Mysovskikh2};
the latter one is also called the M\"oller bound.
\begin{thrm} \label{ft_bound}
\begin{enumerate}
\item Let $X$ be a Euclidean $2e$-design supported by $S$. Then, 
\[
|X| \geq \dim \mathcal{P}_e(S).
\]
\item Let $X$ be a Euclidean $(2e-1)$-design supported by $S$. Then, 
\[
|X| \geq
\left\{\begin{array}{ll}
2 \dim \mathcal{P}_{e-1}^{\ast}(S)-1 &
\qquad \text{if $e$ is odd and $0\in X$}, \\
2 \dim\mathcal{P}_{e-1}^{\ast}(S) &
\qquad \text{otherwise}.
\end{array} \right.
\]
\end{enumerate}
\end{thrm}
\noindent
A Euclidean $t$-design is said to be
{\it tight}
if the equality holds in one of the bounds in Theorem \ref{ft_bound}. 

Hereafter we assume
$G$ is a finite irreducible reflection group in $\mathbb{R}^{n}$.
It is known that finite irreducible reflection
groups are classified completely \cite{Bourbaki}.
Let integers $1= m_1 \leq 
m_2
\leq \cdots \leq m_{n}$ be the exponents of $G$
(see \cite[Ch.V, $\S 6$ ]{Bourbaki}). 
\begin{thrm}[\cite{Goethals-Seidel}] \label{dim_inv}
Let $G$ be a finite irreducible reflection group. Let $q_i$ be the 
dimension of ${\rm Harm}_i(\mathbb{R}^n)^G$.
Then,  
\[
\sum_{i=0}^{\infty}q_i\lambda^i = \prod^{n}_{i=2} 
\frac{1}{1-\lambda^{1+m_i}}.
\]
\end{thrm}
\noindent
Note that for any $x \in \mathbb{R}^n$, the orbit
$x^G$ is a spherical $m_2$-design in $S^{n-1}$ \cite{Goethals-Seidel}.

Let $\alpha_1,\alpha_2, \ldots, \alpha_{n}$ be
the fundamental roots of a reflection group $G$. 
The corner vectors $v_1,v_2,\ldots,v_{n}$ are defined by $v_i
\perp \alpha_j$ if and only if $i \ne j$. 
We may assume $||v_k||=1$. 
We consider the set  

\[ \mathcal{X}(G,J)=\bigcup_{k \in J} r_k v_k^G,
 \]
where
$J\subset \{1,2,\ldots,n\}$ and $r_k>0$. 
Let $R$ denote the set of $r_k$. 

Bajnok \cite{Bajnok}
found new tight Euclidean designs in $\mathcal{X}(B_n,J)$.
In Section 4, using the theory of $G$-invariant harmonic polynomials, 
we extend the method of Bajnok to other reflection groups $G$, and classify the tight Euclidean designs obtained from $\mathcal{X}(G,J)$.  

\section{Xu's theorem}

Let $\Omega = \{ (x, y) \in \mathbb{R}^{2} \mid 0 \leq \sqrt{x^{2} + y^{2}} < \infty \}$.
Let $W$ be a nonnegative weight function on
$[0, \infty)$
with finite moments.
We consider the radial weight function defined by $W(\sqrt{x^2+y^2})$ on $\Omega$.
The following integral is said to be {\it radially symmetric} or {\it circularly symmetric}:
\begin{align*}
\mathcal{I} [f] =&
\int_{\Omega} f (x, y) W(\sqrt{x^{2} + y^{2}}) {\rm d} x {\rm d} y 
\nonumber \\
=&
\int_{0}^{\infty} \bigg(
\int_{0}^{2 \pi} f (r \cos \theta,
r \sin \theta) {\rm d} \theta \bigg) r W(r) {\rm d} r.
\end{align*}
\noindent
To generalize a famous theorem by Verlinden and Cools~\cite{Verlinden-Cools}
on the existence of cubature formula for radially symmetric integral,
Xu proved the following theorem:
\begin{thrm}
{\rm (\cite[Theorem 2.1, Theorem 2.2]{Xu}).}
\label{thm:X98}
{\rm (i)}
Let
\begin{equation}
\label{eq:XuOdd}
\begin{array}{ll}
\mathcal{I}_{2m}(f)   & = \displaystyle \frac{\pi}{m+1} \sum_{i=1}^m \lambda_i \sum_{j=0}^{2m+1} f
\bigg(r_i \cos \frac{(2j+\sigma_i)\pi}{2m+2}, r_i \sin \frac{(2j+\sigma_i)\pi}{2m+2} \bigg), \\
\mathcal{I}_{2m+1}(f) & = \lambda_0f(0,0) + \\
& \quad \displaystyle \frac{\pi}{m+2} \sum_{i=1}^m \lambda_i \sum_{j=0}^{2m+3}
f \bigg(r_i \cos \frac{(2j+\sigma_i)\pi}{2m+4}, r_i \sin \frac{(2j+\sigma_i)\pi}{2m+4} \bigg),
\end{array}
\end{equation}
where $\sigma_i$ takes the value $0$ if $m+i$ is even and the value $1$ if $m+i$ is odd.
Then, $\mathcal{I}_n$ forms a cubature formula of degree $2n-1$ for $\mathcal{I}$
if and only if the following two types of conditions are satisfied:
\begin{align}
\int_0^\infty r^{2j+1} W(r) {\rm d} r & = \sum_{i=1}^{\lfloor n/2 \rfloor} \lambda_i r_i^{2j},
\qquad j = 1, \cdots, n-1, \label{eq:NormOdd} \\
\sum_{i=1}^{\lfloor n/2 \rfloor} \lambda_i r_i^{2j} (-1)^i & =0,
\qquad j = \lfloor (n+3)/2 \rfloor, \cdots, n-1.
\label{eq:CosOdd}
\end{align}
(ii)
With the same symbol $\sigma_i$ as in (i),
let
\begin{align*}
\mathcal{I}_{n}(f)   & = \frac{2\pi}{m+1}
\sum_{i=1}^{\lfloor (n+2)/2 \rfloor} \lambda_i \sum_{j=0}^{2m}
f \bigg(r_i \cos \frac{(2j+\sigma_i)\pi}{2m+1}, r_i \sin \frac{(2j+\sigma_i)\pi}{2m+1} \bigg), 
\end{align*}
where $n = 2m-1$ or $2m$.
Then, $\mathcal{I}_n$ forms a cubature formula of degree $2n$ for $\mathcal{I}$
if and only if the following two types of conditions are satisfied:
\begin{align}
\int_0^\infty r^{2j+1} W(r) {\rm d} r & = \sum_{i=1}^{\lfloor (n+2)/2 \rfloor} \lambda_i r_i^{2j},
\qquad j = 1, \cdots, n,
\label{eq:NormEven} \\
\sum_{i=1}^{\lfloor (n+2)/2 \rfloor}
\lambda_i r_i^{2j} (-1)^i & =0,
\qquad j = \lfloor (n+1)/2 \rfloor, \cdots, n-1.
\label{eq:CosEven}
\end{align}
\end{thrm}

\bigskip

The aim of this section is
to give a proof of Theorem~\ref{thm:X98}
different from the original one by Xu:

\bigskip

\noindent
{\it Proof of Theorem~\ref{thm:X98}}.
Let $D_\ell$ be the dihedral group of order $2\ell$ and
$f(x_1,x_2)$ be a $D_\ell$-invariant polynomial.
Using the polar coordinate system,
we let $f(x_1,x_2) = f(r \cos \theta, r \sin \theta)$.
It is shown that
$f$ can be represented as a polynomial
in two variables $r^2, r^\ell \cos \ell \theta$.
We now consider the case where $n = 2m$ in (i);
the reader will easily see that the same argument as below works in the remaining cases.
The conditions (\ref{eq:NormOdd}), (\ref{eq:CosOdd}) respectively mean
to substitute the monomials $r^2,\cdots,r^{2n-2}$ and
$r^{n+2} \cos {(n+2)\theta}$, $r^{n+4} \cos {(n+2)\theta}$,
$\cdots$, $r^{2n-2} \cos {(n+2)\theta}$
into (\ref{eq:XuOdd}).
Thus the result follows by Theorem~\ref{thm:Sobolev}. $\Box$

\begin{remark}
(i)
Xu's original proof of Theorem~\ref{thm:X98}
is about 2 pages long only in the case of (i) with $n = 2m$.
With this in mind,
we tried a short proof
using the Sobolev theorem.
An advantage of our proof is the simplicity.
Namely,
the Sobolev theorem is the only advanced tool used in our proof, whereas,
Xu's proof requires some technical and advanced tools
in numerical analysis like,
Gaussian-Lobatto quadrature, Gaussian-Radau quadrature.
The proof by Xu also requires some tedious calculations.
Eventually our proof is shorter and simpler than the original one.
Another advantage of our proof:
The conditions (\ref{eq:NormOdd}), (\ref{eq:CosOdd})
(or (\ref{eq:NormEven}), (\ref{eq:CosEven})) are
considerably analytic, and so
researchers in other fields like combinatorics and algebra
will not be familiar with them.
Whereas, our new proof gives a general interpretation
of the above analytic conditions, and
will promise researchers in these areas to enjoy Theorem~\ref{thm:X98} well.
The authors hope that
researchers in many different fields know
the importance of Theorem~\ref{thm:X98} and will be
more interested in classical theories of cubature formulas
developed in numerical analysis.
(ii)
Bajnok~\cite[Theorem 9]{Bajnok1} found a tight Euclidean $t$-design of
$\mathbb{R}^2$ which has the same structure of points as Xu's formula,
as a generalization of a tight $4$-design by
Bannai and Bannai~\cite{Bannai-Bannai}.
To do this, he implicitly used
the same idea as in Theorem~\ref{thm:X98}; for instance
Eq.~(10) in his paper corresponds to Eq.~(\ref{eq:CosOdd})
(or Eq.~(\ref{eq:CosEven})) of our paper.
\end{remark}

\section{Orbits of a reflection group as Euclidean designs}
In this section we classify the tight Euclidean designs obtained from  
$\mathcal{X}(G,J)$ for a finite irreducible reflection group $G$.
A finite set $X \subset \mathbb{R}^d$ is said to be {\it antipodal} if $X=-X$.  
A tight Euclidean $2e$-design has a weight function which 
is constant on each $X_i$~\cite{Bannai-Bannai}, and so does an antipodal
tight Euclidean
$(2e-1)$-design~\cite{Etsuko2}. 
Throughout this section
we assume a weight function is constant on each $G$-orbit.

First, we look at a stronger theorem than Theorem \ref{thm:NS88}
for $G$-invariant Euclidean designs.
A Euclidean $t$-design $X$ is said to be
{\it $G$-invariant} if $X$ is a union of $G$-orbits and
to each point of the same orbit an equal weight is assigned.
\begin{lem} \label{constant}
Let $G$ be a subgroup of $O(\mathbb{R}^n)$.
Let $f$ be a $G$-invariant polynomial and
$x^G$ be a $G$-orbit.
Then, $f(y)=f(z)$ for any $y,z\in x^G$.
\end{lem} 
\begin{proof}
Straightforward.
\end{proof} 
\
Let $|G|$ be the order of a group $G$. 
\begin{thrm} \label{inv}
Let $G$ be a subgroup of $O(\mathbb{R}^n)$.
Let $X=\cup_{k=1}^M r_k x_k^G$,
where $x_k\in S^{n-1}$ and $r_k>0$.
The following are equivalent: 
\begin{enumerate}
\item $X$ is a $G$-invariant Euclidean $t$-design. 
\item $\sum_{x \in X} w(x) ||x||^{2 j} \varphi(x) = 0$
for any $\varphi \in {\rm Harm}_l (\mathbb{R}^n)^G$ with
$1 \leq l \leq t, 0 \leq j \leq \lfloor \frac{t-l}{2} \rfloor$. 
\end{enumerate}
\end{thrm} 
\begin{proof}
For $f \in {\rm Harm}_l(\mathbb{R}^d)$, the polynomial 
\[
\varphi(\xi)=\frac{1}{|G|}\sum_{g \in G}f(\xi^g)
\]
is an element of ${\rm Harm}_l(\mathbb{R}^d)^G$.
Let $w(x)=w_k$ for every $x \in x_k^G$.
By Lemma \ref{constant}, for any $f \in {\rm Harm}_l(\mathbb{R}^d)$, we have
\begin{equation*}
\begin{split}
\sum_{x \in X} w(x) ||x||^{2 j} f(x)& = \sum_{k=1}^M w_k r_k^{2 j} \sum_{x \in x_k^G} f(r_k x) \\
									& = \sum_{k=1}^M \frac{w_k r_k^{2 j+l} |x_k^G|}{|G|} \sum_{g \in G} f(x_k^g) \\
									& = \sum_{k=1}^M w_k r_k^{2 j+l} |x_k^G| \varphi(x_k) \\
									& = \sum_{k=1}^M w_k r_k^{2 j} \sum_{x \in x_k^G} \varphi(r_k x) \\ 
									& = \sum_{x \in X} w(x) ||x||^{2 j} \varphi(x). 
\end{split}
\end{equation*}
The result thus follows by Theorem \ref{thm:NS88}. 
\end{proof}

\begin{remark}
The radii $r_k$ are not necessarily mutually distinct in Theorem \ref{inv}. 
Goethals and Seidel \cite{Goethals-Seidel} stated
Theorem \ref{thm:Sobolev} for spherical designs.
Theorem \ref{inv} with all $r_k=1$ means
the theorem of $G$-invariant spherical designs.
The approach of using the orbits under subgroups of $O(\mathbb{R}^d)$
has long been considered \cite{Bannai-Bannai2}.
Theorem~\ref{inv} reduces the computational cost to check
the strength of a $G$-invariant Euclidean design
less than using Theorem~\ref{thm:NS88}.
\end{remark}

From now on,
let $G$ be a finite irreducible reflection group,
and $v_k$ be a corner vector.
Put $N_k=|v_k^G|$. The following is immediate by
Lemma \ref{constant} and Theorem \ref{inv}.  
\begin{cor} \label{main}
$\mathcal{X}(G,J)$ is a Euclidean $t$-design if and only if 
there exist $w_k>0$ and $r_k>0$ such that
the equation $\sum_{k \in J}w_k r_k^{2j+i} N_k f(v_k)=0$ holds for any $f 
\in {\rm Harm}_i(\mathbb{R}^d)^G$,
where 
$i,j$ are positive integers satisfying $1 \leq 2j+i \leq t$. 
\end{cor} 
The dimension of ${\rm Harm}_i(\mathbb{R}^d)^G$
is clear by Theorem \ref{dim_inv}. 
We can determine the basis of ${\rm Harm}_i(\mathbb{R}^d)^G$ by
harmonic polynomials $f$ satisfying $f(x^g)=f(x)$ for 
each generator $g$ of $G$. 
By the basis of ${\rm Harm}_i(\mathbb{R}^d)^G$ and Corollary \ref{main}, 
we know a necessary and sufficient condition for
$\mathcal{X}(G,J)$ to be a Euclidean $t$-design.
Bajnok \cite{Bajnok} found an explicit such condition
for the group $B_n$ by using Corollary \ref{main}. 
For other groups, it is possible to give the conditions,
but the statements are not simple. 
Therefore we do not write them in the present paper. 

Now,
let us classify 
the tight Euclidean design obtained from $\mathcal{X}(G,J)$. 
For each group, we determine the possible maximum strength of $\mathcal{X}(G,J)$ for any $J$ and radii $R$.
Since the cardinality of $v_k^G$ is easily calculated, 
we can check whether the total size of a union of several orbits attains the fisher type inequality. 
For the set attaining the bound, we give its maximum strength by Corollary \ref{main}.

Hereafter let $e_i\in \mathbb{R}^n$ be the row vector whose 
$i$-th entry is $1$ and other entries are $0$. 
Let $S_n$ be the symmetric group. Define 
\[{\rm 
sym}(f):=\frac{1}{|(S_n)_{f}|}\sum_{g \in S_n} f(x^g), 
\]
 where $(S_n)_f:=\{g \in S_n \mid f(x^g)=f(x) \}$.

\subsection{Group $A_n$}
\quad \\
\textbf{Dynkin diagram}
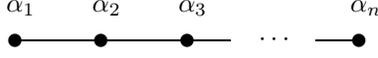
\begin{center}
\unitlength.015in
\begin{picture}(180,20)
\put( 30, 10){\circle*{5}}
\put( 30, 10){\line(1,0){30}}
\put( 60, 10){\circle*{5}}
\put( 60, 10){\line(1,0){30}}
\put( 90, 10){\circle*{5}}
\put( 90, 10){\line(1,0){15}}
\put( 115, 8 ){$\cdots$}
\put( 135, 10){\line(1,0){15}}
\put( 150, 10 ){\circle*{5}}

\put( 27, 20){$\alpha_1$}
\put( 57, 20){$\alpha_2$}
\put( 87, 20){$\alpha_3$}
\put( 147, 20){$\alpha_n$}

\end{picture}
\end{center}

\noindent
\textbf{Exponents} \\
$1,2, \ldots, n$

\noindent
\textbf{Fundamental roots} \\
$\alpha_{i}:=e_i-e_{i+1}$ for $1\leq i \leq n-1$. 
$\alpha_n:=[a,a,\ldots,a,b]$ where 
$a= (-1+\sqrt{n+1})/n$ and $b=(n-1+\sqrt{n+1})/n$. 

\noindent
\textbf{Corner Vectors}\\ 
$v_k=[c_k, \ldots ,c_k,d_k, \ldots , d_k]$ whose first $k$ coordinates 
are equal to $c_k$, and last $n-k$ coordinates are equal to $d_k$, where 
	\[ 
	c_k=\frac{n+1-k+\sqrt{n+1}}{\sqrt{k(n+1-k)(n+2+2\sqrt{n+1})}}, \qquad 
d_k=\frac{-k}{\sqrt{k(n+1-k)(n+2+2\sqrt{n+1})}}. 
	\]

\noindent
\textbf{Reflection group} \\
The reflection group $A_n \subset O(\mathbb{R}^n)$ is generated by the 
following: 
\[
r(\alpha_i)=\left[
{}^te_1, 
\cdots 
{}^te_{i-1}, 
{}^te_{i+1}, 
{}^te_{i},  
{}^te_{i+2}, 
\cdots ,
{}^te_n    
\right] \text{ for $1\leq i \leq n-1$},
\]

\[
r(\alpha_n)=\left[
\begin{array}{cc}
I_{n-1}-a^2 
({}^tjj)
& -ab {}^tj\\
-ab  j & 1-b^2
\end{array}
\right], 
\]
where $I_{n}$ is the identity matrix of size $n$, 
and $j$ is the all-ones row vector. 

\noindent
\textbf{Orbits} \\
Let $U_1$ be the set of all vectors such that
$k$ coordinates are equal to $c_k$, 
and other $n-k$ coordinates are equal to $d_k$. 
Let $U_2$ be the set of all vectors such that
 $k-1$ coordinates are equal to 
$c_k'$,
 and other $n+1-k$ coordinates are equal to $d_k'$, where 
	\[ 
	c'_k= \frac{n+1-k}{\sqrt{k(n+1-k)(n+2+2\sqrt{n+1})}}, \qquad  
d'_k=\frac{-k-\sqrt{n+1}}{\sqrt{k(n+1-k)(n+2+2\sqrt{n+1})}}. 
	\] 
Then the orbit $v^{A_n}_k= U_1 \cup U_2$.  
 Furthermore, we have $N_k=\binom{n+1}{k}$
and $v_k^{A_n}=-v_{n+1-k}^{A_n}$.

\noindent
\textbf{Harmonic Molien series}
\begin{equation*}
\frac{1}{(1-t^3)(1-t^4)\cdots (1-t^{n+1})}=
\left\{\begin{array}{ll}
1+t^3+t^6+ \cdots, & \qquad \text{ if } n = 2, \\
1+t^3+t^4+t^6+\cdots, & \qquad \text{ if } n =3, \\ 
1+t^3+t^4+t^5+\cdots, & \qquad \text{ if } n \geq 4. \\
\end{array} \right.
\end{equation*}
\noindent
\textbf{$G$-invariant harmonic polynomials}\\ 
1. {\it Degree $3$.} \\ 
Note that $\dim {\rm Harm}_3(\mathbb{R}^n)^{A_n} =1$ for any $n \geq 2$. 
${\rm Harm}_3(\mathbb{R}^n)^{A_n}$ is spanned by the following: \\
(i) $n=2$.
\begin{equation*}
f_3=x_1^3-3 x_1^2 x_2-3 x_1 x_2^2+x_2^3.
\end{equation*} 
(ii) $n=3$.
\begin{equation*}
f_3={\rm sym}(x_1^3)-\frac{3}{2} {\rm sym}(x_1 x_2^2)-
\frac{3}{4} {\rm sym}(x_1 x_2 x_3).
\end{equation*}
(ii)
$n\geq 4$.
\begin{equation*}
f_3={\rm sym}(x_1^3)-\frac{3}{n-1} {\rm sym}(x_1 x_2^2)+
\frac{6(2-\sqrt{n+1})}{(n-1)(n-3)} {\rm sym}(x_1 x_2 x_3).
\end{equation*}
2. {\it Degree $4$.}\\
Note that $\dim {\rm Harm}_4(\mathbb{R}^2)^{A_2} =0$ and $\dim {\rm 
Harm}_4(\mathbb{R}^n)^{A_n} =1$ for any $n \geq 3$. The following are $S_n$-invariant harmonic polynomials: 
\begin{align*}
&h_{4,1}= {\rm sym}(x_1^4)-\frac{6}{n-2} {\rm sym}(x_1^2 x_2^2), \\ 
&h_{4,2}= {\rm sym}(x_1 x_2 x_3 x_4), \\
&h_{4,3}= {\rm sym}(x_1 x_2^3)-\frac{6}{n-2}{\rm sym}(x_1 x_2 x_3^2).
\end{align*}
${\rm Harm}_4(\mathbb{R}^n)^{A_n}$ is spanned by the following:\\
(i) $n=3$.
\begin{equation*}
f_4 =  h_{4,1} - \frac{20}{13} h_{4,3}.
\end{equation*}
(ii) $n \geq 4$. 
\begin{equation*}
f_4= h_{4,1}+c_{4,2}h_{4,2}+c_{4,3}
h_{4,3},
\end{equation*}
where
\begin{align*}
& c_{4,2}=\frac{24 (n+2) (n^2-5n-12+4\sqrt{n+1})}{(n-1)(n-2)(n^3-2 n^2 -15 
n -16)}, \\
& c_{4,3}=-\frac{4 (n+2)(n^2-2n-7-(n-1)\sqrt{n+1})}{(n-1)(n^3-2 n^2 -15 n 
-16)}.
\end{align*}
3. {\it Degree $5$}.\\
Note that $\dim {\rm Harm}_5(\mathbb{R}^n)^{A_n} =0$ for $n=2,3$, and $\dim 
{\rm Harm}_5(\mathbb{R}^n)^{A_n} =1$ for any $n \geq 4$. 
The following are $S_n$-invariant harmonic polynomials:  
\begin{align*}
&h_{5,1}= {\rm sym}(x_1^5)-\frac{10}{n-1} {\rm sym}(x_1^2 
x_2^3)+\frac{30}{(n-1)(n-2)}{\rm sym}(x_1 x_2^2 x_3^2), \\ 
&h_{5,2}= {\rm sym}(x_1^5)-\frac{10}{n-1} {\rm sym}(x_1^2 
x_2^3)+\frac{5}{n-1} {\rm sym}(x_1 x_2^4), \\
&h_{5,3}= {\rm sym}(x_1 x_2 x_3^3)-\frac{9}{n-3}{\rm sym}(x_1 x_2 x_3 
x_4^2), \\
&h_{5,4}= {\rm sym}(x_1 x_2 x_3 x_4 x_5).
\end{align*} 
${\rm Harm}_4(\mathbb{R}^n)^{A_n}$ is spanned by the following: \\
(i) $n=4$.
\begin{equation*} 
f_5=h_{5,1}+\frac{17-20\sqrt{5}}{58} h_{5,2}+\frac{10(18+\sqrt{5})}{87} 
h_{5,3}.
\end{equation*}
(ii) $n \geq 5$. 
\begin{equation*}
f_5=h_{5,1}+c_{5,2} h_{5,2}+c_{5,3} h_{5,3}+c_{5,4}h_{5,4},
\end{equation*}
where
\begin{align*}
&c_{5,2}=-\frac{2n^3+5n^2-21n-90-n(n+6)\sqrt{n+1}}{4n^3+3n^2-60n-180}, \\
&c_{5,3}=\frac{20(2n^3+6n^2-32n-168+(n^2-8n+12)\sqrt{n+1})}{(n-1)(n-2)(4n^3+3n^2-60n-180)}, 
\\
&c_{5,4}=-\frac{120(n+6)(n^2-11n-78+(2n^2-2n+12)\sqrt{n+1})}{(n-1)(n-2)(n-3)(4n^3+3n^2-60n-180)}.
\end{align*}

\noindent
\textbf{Substitute $v_k$ for $G$-invariant harmonic polynomials }\\ 
1. {\it Degree $3$.} \\ 
For $n=2,n\geq 4$, 
\begin{equation*}
f_3(v_k)=-\frac{k-\frac{n+1}{2}}{\sqrt{k(n+1-k)}} \phi_3(n),
\end{equation*} 
where
\begin{equation*}
\phi_3(n)=\frac{2(n^3+3n^2-12n-16+(3n^2-4n-16)\sqrt{n+1})}{(n-1)(n-3)(n+2+2\sqrt{n+1})^{\frac{3}{2}}}.
\end{equation*}
For $n=3$, 
\begin{equation*}
f_3(v_1)=\frac{729}{4}, \qquad f_3(v_2)=0, \qquad f_3(v_3)=-\frac{729}{4}. 
\end{equation*}
2. {\it Degree $4$.}\\
For $n\geq 3$, 
\begin{equation*}
f_4(v_k)= \frac{(k-\alpha )(k-\beta )}{k(n+1-k)}\phi_4(n),
\end{equation*}
where
\begin{align*}
& \phi_4(n) \\
& = \frac{6(n+1)(n^5+7n^4-24n^3-160n^2-256n-128
+4(n^4-20n^2-48n-32)\sqrt{n+1})}{
(n-1)(n-2)(n^3-2n^2-15n-16)(n+2+2\sqrt{n+1})^2}, 
\end{align*}
\begin{equation*}
\alpha = \frac{n+1}{2}-\frac{\sqrt{3(n^2-1)}}{6}, \qquad \beta = 
\frac{n+1}{2}+\frac{\sqrt{3(n^2-1)}}{6}.
\end{equation*}
3. {\it Degree $5$.}\\
For $n\geq 4$, 
\begin{equation*}
f_5(v_k)=-\frac{(k-\frac{n+1}{2})(k-\alpha')(k-\beta')}{(k(n+1-k))^{\frac{3}{2}}}\phi_5(n),
\end{equation*}
where
\begin{multline*}
\phi_5(n)=
 \frac{24(n+1)(2n^6+31n^5+50 n^4-448 n^3-2144 n^2-3200 
n-1536)}{(n-1)(n-2)(n-3)(4 n^3+3 n^2-60 n-180)(n+2+2\sqrt{n+1})^{\frac{5}{2}}} \\ 
 +\frac{24(n+1)(11 n^5+50 n^4-96 n^3 -1120 n^2-2432 n -1536  
)\sqrt{n+1}}{(n-1)(n-2)(n-3)(4 n^3+3 n^2-60 n-180)(n+2+2\sqrt{n+1})^{\frac{5}{2}}},
\end{multline*} 
\begin{equation*}
\alpha'=\frac{n+1}{2}-\frac{\sqrt{3(n+1)(2n-3)}}{6}, \qquad 
\beta'=\frac{n+1}{2}+\frac{\sqrt{3(n+1)(2n-3)}}{6}.
\end{equation*}

\begin{thrm} \label{A_n}
There is no choice of $J$, $R$ and $w$ for which $(\mathcal{X}(A_n,J),w)$ 
is a Euclidean $6$-design. 
\end{thrm}

\begin{proof}
The polynomial of degree $6$
\[
f(x_1,x_2,\ldots, x_n)={\rm sym}(x_1 x_2^5)-\frac{10}{3}{\rm sym}(x_1^3 
x_2^3)
\]
is harmonic for any $n\geq 2$. We can calculate 
\[
\sum_{x \in v_k^{A_n}} f(x) =  g_1(n,k)( k(k-n-1)(4k^2-4(n+1)k+n^2+5n+4) 
g_2(n)+g_3(n)),
\]  
where 
\begin{align*}
& g_1(n,k)= 
\frac{n(n+1)}{3k^3(n+1-k)^3(n+2+2\sqrt{n+1})^3}\binom{n-1}{k-1}, \\
& g_2(n)=5(n^2+11n+12+(6n+12)\sqrt{n+1}), \\
& g_3(n)=(n+1)^2(n+2)(2n^2+28n+30+(15n+30)\sqrt{n+1}).
\end{align*}
Note $g_1(n,k)>0$ for $1\leq k\leq n$. Define
\[
F(k)=k(k-n-1)(4k^2-4(n+1)k+n^2+5n+4) g_2(n)+g_3(n).
\]
For a fixed $n$, we prove $F(k) < 0$ for $1\leq k \leq n$. 
We have 
\[
\frac{d}{d k}
F(k)=16\left(k-\frac{n+1}{2} \right) (k-\alpha'' 
)(k-\beta'') g_2(n),
\]
where
\[
\alpha''= \frac{n+1}{2}-\frac{\sqrt{2(n+1)(n-2)}}{4}, \qquad \beta''= 
\frac{n+1}{2}+\frac{\sqrt{2(n+1)(n-2)}}{4}. 
\]
If $F(1)=F(n)<0$, $F((n+1)/2)<0$, and $F(\alpha'')=F(\beta'')<0$, then 
$F(k)<0$ for all $1\leq k \leq n$. Indeed for $n \geq 2$, 
\begin{align*}
&F(1)=F(n) \\
& \qquad = -(n-1)(3n^4+27n^3+10n^2+26n+60+(15n^3+15n^2+60)\sqrt{n+1})<0, \\
&F(\alpha'' )=F(\beta'' ) \\
& \qquad =  -\frac{1}{16} 
n(n+1)^2(5n^3+63n^2+68n-16+(30n^2+60n)\sqrt{n+1})<0, \\
&F\left( \frac{n+1}{2} \right) =  -\frac{1}{4} 
(n-1)(n+1)^2(7n^2+59n+60+(30n+60)\sqrt{n+1})<0. 
\end{align*}
Therefore $\sum_{k \in J}\sum_{x \in v_k^{A_n}} w_k r_k f(x) <0$ for any 
$J$, $R$ and $w$. 
\end{proof}

\begin{thrm} \label{thm:A_n}
$\mathcal{X}(A_n,J)$ is not a tight Euclidean $t$-design
except for the sets
in Table 1.
\end{thrm} 
\begin{proof}
We prove only the classification of tight Euclidean $4$-designs on 
two concentric spheres obtained from $\mathcal{X}(A_n,J)$. 
The other cases can be proved by a similar way.

$f_3(v_k)$ and $\phi_3(n)$ are defined as above. 
Since $\phi_3(n)>0$ for $n>1$, 
$f_3(v_k)=0$ if and only if $n\equiv 1 \pmod 2$ and $k=(n+1)/2$. 
Clearly $f_3(v_k)>0$ for $k<(n+1)/2$, and $f_3(v_k)<0$ for $k>(n+1)/2$. 
Therefore $J$
must contain $k_1$ and $k_2$ such that $k_1<(n+1)/2<k_2$ by 
Corollary \ref{main}.  

The size of a tight Euclidean $4$-design on two concentric spheres is $(n+1)(n+2)/2$. 
By noting that $(n+1)(n+2)/2<N_3=N_{n-2}$ for $n > 5$, 
we can determine $J=\{1,n-1\}$, (or equivalently $J=\{2,n\}$) for any $n>2$, or $J=\{1,2\}$ for $n=2$. 
For $n=2$, we can obtain tight Euclidean $4$-designs on two concentric spheres as in Table 1. 

$f_4(v_k)$, $\phi_4(n)$, $\alpha$, and $\beta$ are defined as above. 
Note that $\phi_4(n)\ne 0$, $\alpha>1$, and $\beta<n$ for any integer $n>2$. 
Therefore $1<\alpha<n-1<\beta$, (or equivalently $\alpha<2<\beta<n$) holds   
by Corollary \ref{main}. 
The integers satisfying the condition are only $n=4,5,6$. 
For $n=4,5,6$, we can obtain tight Euclidean $4$-designs on two concentric spheres as in Table 1.  
\end{proof}
\begin{table} 
\begin{center}
\begin{tabular}{|c|c|c|c|c|c|}
\hline
$|R|$& $t$ & $n$ & $J$     & $r_i$ & $w_i$ \\
\hline
$1$& $2$ & any & $\{1\}$   & $r_1=1$  & $w_1=1$ \\
 &$2$ & any & $\{n\}$   &  $r_n=1$    & $w_n=1$ \\
&$3$ & $3$ & $\{2\}$   &  $r_2=1$     & $w_2=1$ \\
&$5$ & $2$ & $\{1,2\}$ &  $r_1=r_2=1$ & $w_1=w_2=1$ \\
&$5$ & $7$ & $\{2,6\}$  & $r_2=r_6=1$ & $w_2=w_6=1$ \\
\hline
$2$ & $4$ & $2$ &$\{1,2\}$ & $r_1=1$, $r_2\ne 1$ & $w_1=1$, 
$w_2=\frac{1}{r_2^3}$ \\
    & $4$ & $4$ &$\{1,3\}$ &$r_1=1$, $r_3=\frac{1}{\sqrt{6}}$& $w_1=1$, 
$w_3=27$  \\   
    & $4$ & $4$ &$\{2,4\}$ &$r_4=1$, $r_2=\frac{1}{\sqrt{6}}$& $w_4=1$, 
$w_2=27$  \\    
    & $4$ & $5$ &$\{1,4\}$ &$r_1=1$, $r_4=\sqrt{\frac{8}{5}}$& $w_1=1$, 
$w_4=\frac{1}{2}$  \\  
    & $4$ & $5$ &$\{2,5\}$ &$r_5=1$, $r_2=\sqrt{\frac{8}{5}}$& $w_5=1$, 
$w_2=\frac{1}{2}$  \\ 
    & $4$ & $6$ &$\{1,5\}$ &$r_1=1$, $r_5=\sqrt{15}$& $w_1=1$, 
$w_5=\frac{1}{81}$  \\  
    & $4$ & $6$ &$\{2,6\}$ &$r_6=1$, $r_2=\sqrt{15}$& $w_6=1$, 
$w_2=\frac{1}{81}$  \\  
    & $5$ & $3$ &$\{1,2,3\}$& $r_1=r_3=1$, $r_2\ne 1$ & $w_1=w_3=1$, 
$w_2=\frac{9}{8r_2^4}$\\
 & $5$ & $5$ &$\{1,3,5\}$& $r_1=r_5=1$, $r_3 \ne 1$ & $w_1=w_5=1$, 
$w_3=\frac{27}{25r_2^4}$\\   
\hline 
\end{tabular}
\caption{Tight Euclidean $t$-designs from $\mathcal{X}(A_n,J)$ }
\end{center}
\end{table}

\subsection{Group $B_n$}
\quad \\
\textbf{Dynkin diagram}\begin{center}
\unitlength.015in
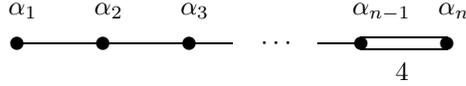
\begin{picture}(180,20)
\put( 30, 10){\circle*{5}}
\put( 30, 10){\line(1,0){30}}
\put( 60, 10){\circle*{5}}
\put( 60, 10){\line(1,0){30}}
\put( 90, 10){\circle*{5}}
\put( 90, 10){\line(1,0){15}}
\put( 115, 8 ){$\cdots$}
\put( 135, 10){\line(1,0){15}}
\put( 150, 10 ){\circle*{5}}
\put( 150, 12){\line(1,0){30}}
\put( 150, 8){\line(1,0){30}}
\put( 180, 10 ){\circle*{5}}

\put( 27, 20){$\alpha_1$}
\put( 57, 20){$\alpha_2$}
\put( 87, 20){$\alpha_3$}
\put( 147, 20){$\alpha_{n-1}$}
\put( 177, 20){$\alpha_{n}$}
\put( 162, -3){4}

\end{picture}
\end{center}

\noindent
\textbf{Exponents} \\
$1,3, \ldots, 2n-1$

\noindent
\textbf{Fundamental roots} \\
$\alpha_{i}:=e_i-e_{i+1}$ for $1\leq i \leq n-1$ and $\alpha_n:=\sqrt{2}e_n$.

\noindent
\textbf{Corner Vectors} \\
$v_k=[1/\sqrt{k}, \ldots ,1/\sqrt{k},0, 
\ldots , 0]$, where $v_k$ has $k$ coordinates equal to $1/\sqrt{k}$.

\noindent
\textbf{Reflection group} \\
The reflection group $B_n \subset O(\mathbb{R}^n)$ is generated by the 
following: 
\[
r(\alpha_i)=\left[
{}^te_1, 
\cdots 
{}^te_{i-1}, 
{}^te_{i+1}, 
{}^te_{i},  
{}^te_{i+2}, 
\cdots ,
{}^te_n    
\right] \text{ for $1\leq i \leq n-1$},
\]
\[
r(\alpha_n)=\left[
{}^te_1, 
\cdots 
{}^t e_{n-1},
-{}^te_n    
\right].
\]

\noindent
\textbf{Orbits}\\ The orbit $v_k^{B_n}$ is the set of vectors with exactly 
$k$ nonzero coordinates equal to $\pm 1/\sqrt{k}$. Note that $v_k^{B_n}$ 
is antipodal and $N_k=2^k \binom{n}{k}$. 

\noindent
\textbf{Harmonic Molien series}
\begin{eqnarray*}
\frac{1}{(1-t^4)(1-t^6)\cdots (1-t^{2n})}=
\left\{\begin{array}{ll}
1+t^4+t^8+ \cdots, & \qquad \text{ if } n = 2, \\
1+t^4+t^6+ \cdots, & \qquad \text{ if } n\geq 3. 
\end{array} \right.
\end{eqnarray*}

\noindent
\textbf{$G$-invariant harmonic polynomials}\\
1. {\it Degree $4$.} \\ 
Note that $\dim({\rm Harm}_4(\mathbb{R}^n)^{B_n})=1$ for any $n \geq 2$. 
The following is a $B_n$-invariant harmonic polynomial of degree $4$: 
\begin{eqnarray*}
f_4={\rm sym}(x_1^4)- \frac{6}{n-1}{\rm sym}(x_1^2 x_2^2).
\end{eqnarray*}
2. {\it Degree $6$.} \\ 
Note that $\dim({\rm Harm}_6(\mathbb{R}^2)^{B_2})=0$ and $\dim({\rm 
Harm}_6(\mathbb{R}^n)^{B_n})=1$ for any $n \geq 3$. The following is a 
$B_n$-invariant harmonic polynomials of degree $6$: 
\begin{eqnarray*}
f_6={\rm sym}(x_1^6)- \frac{15}{n-1}{\rm sym}(x_1^2 x_2^4)+ 
\frac{180}{(n-1)(n-2)}{\rm sym}(x_1^2 x_2^2 x_3^2).
\end{eqnarray*}
\textbf{Substitute $v_k$ for $G$-invariant harmonic polynomials}\\
1. {\it Degree $4$.}  
\begin{eqnarray*}
f_4(v_k)=\frac{1}{k}\left(1-3 \frac{k-1}{n-1} \right).
\end{eqnarray*}
2. {\it Degree $6$.} 
\begin{eqnarray*}
f_6(v_k)=\frac{1}{k^2}\left(1-15 \frac{k-1}{n-1}+30 
\frac{(k-1)(k-2)}{(n-1)(n-2)} \right).
\end{eqnarray*}

\begin{thrm}[\cite{Bajnok}] \label{baj1}
There is no choice of $R$, $J$, and $w$ for which $(\mathcal{X}(B_n,J),w)$ 
is a Euclidean $8$-design. 
\end{thrm}
\begin{thrm}[\cite{Bajnok}] \label{baj2}
$\mathcal{X}(B_n,J)$ is not a tight Euclidean $t$-design
except for the sets in Table 2.
\end{thrm}
\begin{remark}
We can also prove Theorems \ref{baj1} and \ref{baj2} by the $B_n$-invariant 
harmonic polynomials. 
\end{remark}

\begin{table}
\begin{center}
\begin{tabular}{|c|c|c|c|c|c|}
\hline
$|R|$& $t$ & $n$ & $J$     & $r_i$ & $w_i$ \\
\hline
$1$  & $3$ & any &$\{1\}$  & $r_1=1$ & $w_1=1$ \\
	 & $3$ & $2$ & $\{2\}$ & $r_2=1$ & $w_2=1$ \\
	 & $7$ & $2$ &$\{1,2\}$& $r_1=r_2=1$& $w_1=w_2=1$\\
\hline
$2$  & $5$ & $2$ &$\{1,2\}$& $r_1=1$, $r_2 \ne 1$& $w_1=1$, 
$w_2=\frac{1}{r_2^4}$\\
	 & $5$ & $3$ &$\{1,3\}$& $r_1=1$, $r_3 \ne 1$&$w_1=1$, 
$w_3=\frac{9}{8r_2^4}$\\
	 & $7$ & $4$ &$\{1,2,4\}$&$r_1=r_4=1$, $r_2\ne 1$ &$w_1=w_4=1$, 
$w_2=\frac{1}{r_2^6}$\\
\hline
$3$  & $7$ & $3$ &$\{1,2,3\}$& $r_1=1$, 
$r_2=\sqrt{\frac{2r_3^2}{5r_3^2-3}}$, ($r_i\ne r_j$) & $w_1=1$, $w_2=\frac{4}{5r_2^6}$, $w_3=\frac{27}{40 
r_3^6}$\\
\hline
\end{tabular}
\caption{Tight Euclidean $t$-designs from $\mathcal{X}(B_n,J)$ }
\end{center}
\end{table}

\subsection{Group $D_n$}
\quad \\
\textbf{Dynkin diagram}\begin{center}
\unitlength.015in
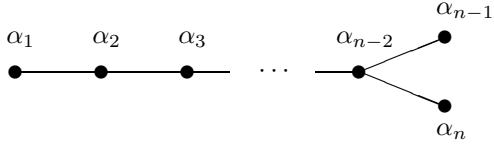
\begin{picture}(180,20)
\put( 30, 10){\circle*{5}}
\put( 30, 10){\line(1,0){30}}
\put( 60, 10){\circle*{5}}
\put( 60, 10){\line(1,0){30}}
\put( 90, 10){\circle*{5}}
\put( 90, 10){\line(1,0){15}}
\put( 115, 8 ){$\cdots$}
\put( 135, 10){\line(1,0){15}}
\put( 150, 10 ){\circle*{5}}
\put( 150, 10){\line(5,-2){30}}
\put( 150, 10){\line(5,2){30}}
\put( 180, 22 ){\circle*{5}}
\put( 180, -2 ){\circle*{5}}

\put( 27, 20){$\alpha_1$}
\put( 57, 20){$\alpha_2$}
\put( 87, 20){$\alpha_3$}
\put( 142, 20){$\alpha_{n-2}$}
\put( 177, 30){$\alpha_{n-1}$}
\put( 177, -12){$\alpha_{n}$}

\end{picture}
\end{center}

\noindent
\textbf{Exponents} \\
$1,3, \ldots, 2n-3,n-1$

\noindent
\textbf{Fundamental roots} \\
$\alpha_{i}:=e_i-e_{i+1}$ for $1\leq i \leq n-1$ and 
$\alpha_n:=e_{n-1}+e_n$.

\noindent
\textbf{Corner Vectors} \\ $v_k=[1/\sqrt{k}, \ldots ,1/\sqrt{k},0, \ldots , 
0]$, where $v_k$ has $k$ coordinates equal to $1/\sqrt{k}$ for $1\leq k \leq 
n-2$. 
$v_{n-1}=[1/\sqrt{n},1/\sqrt{n},\ldots, 1/\sqrt{n},-1/\sqrt{n}]$ and 
$v_{n}=[1/\sqrt{n},1/\sqrt{n},\ldots, 1/\sqrt{n}]$.

\noindent
\textbf{Reflection group} \\
The reflection group $D_n \subset O(\mathbb{R}^n)$ is generated by the 
following: 
\[
r(\alpha_i)=\left[
{}^te_1, 
\cdots 
{}^te_{i-1}, 
{}^te_{i+1}, 
{}^te_{i},  
{}^te_{i+2}, 
\cdots ,
{}^te_n    
\right] \text{ for $1\leq i \leq n-1$},
\]
\[
r(\alpha_n)=\left[
{}^te_1, 
\cdots 
{}^t e_{n-2},
-{}^te_n,
-{}^te_{n-1}    
\right].
\]

\noindent
\textbf{Orbits}\\
For $1 \leq k \leq n-2$, $v_k^{D_n}=v_k^{B_n}$. 
The orbit $v_n^{D_n}$ (resp.\ $v_{n-1}^{D_n}$) consists of the vectors 
$\{\pm 1/\sqrt{n}\}^n$ with an even (resp.\ odd) number of negative 
coordinates. 
Note that 
$v_{n}^{D_n} = -v_{n-1}^{D_n}$ for odd $n$, and both $v_{n}^{D_n}$ and 
$v_{n-1}^{D_n}$ 
are antipodal for even $n$. Furthermore, 
$|N_{n-1}|=|N_{n}|=2^{n-1}$. 

\noindent
\textbf{Harmonic Molien series}
\begin{eqnarray*}
\frac{1}{(1-t^4)(1-t^6)\cdots (1-t^{2n-2})(1-t^n)}=
\left\{\begin{array}{ll} 
1+2 t^4+t^6+ 3 t^8 + \cdots, &  \text{ if } n = 4, \\
1+t^4+t^5+t^6+2 t^8+ \cdots, &  \text{ if } n = 5, \\
1+t^4+2 t^6+2 t^8+ \cdots, &    \text{ if } n = 6, \\
1+t^4+t^6 + t^7+2t^8  \cdots, & \text{ if } n = 7, \\
1+t^4+t^6 +2t^8  \cdots, & \text{ if } n  \geq 8. \\
\end{array} \right.
\end{eqnarray*}

\noindent
\textbf{$G$-invariant harmonic polynomials}\\
1. {\it Degree $4$.} \\ 
Note that $\dim({\rm Harm}_4(\mathbb{R}^4)^{D_4})=2$ and $\dim({\rm 
Harm}_4(\mathbb{R}^n)^{D_n})=1$ for any $n \geq 5$. The following are 
$D_n$-invariant harmonic polynomials of degree $4$: \\
\begin{eqnarray*}
f_4={\rm sym}(x_1^4)- \frac{6}{n-1}{\rm sym}(x_1^2 x_2^2).
\end{eqnarray*}
The following is a $D_4$-invariant harmonic polynomial of degree $4$, which 
is linearly independent of $f_4$: 
\begin{eqnarray*}
f_{4,2} = x_1 x_2 x_3 x_4.
\end{eqnarray*}

\noindent
2. {\it Degree $5$.} \\
Note that $\dim({\rm Harm}_5(\mathbb{R}^5)^{D_5})=1$ and $\dim({\rm 
Harm}_5(\mathbb{R}^n)^{D_n})=0$ for any $n \ne 5$. The following is a 
$D_5$-invariant harmonic polynomial of degree $5$: 
\begin{eqnarray*}
f_5=x_1 x_2 x_3 x_4 x_5. 
\end{eqnarray*}

\noindent
3. {\it Degree $6$.} \\ 
Note that $\dim({\rm Harm}_6(\mathbb{R}^6)^{D_6})=2$ and $\dim({\rm 
Harm}_6(\mathbb{R}^n)^{D_n})=1$ for any $n \ne 6$.
The following is a $D_n$-invariant harmonic polynomial of degree $6$: \\
\begin{eqnarray*}
f_6={\rm sym}(x_1^6)- \frac{15}{n-1}{\rm sym}(x_1^2 x_2^4)+ 
\frac{180}{(n-1)(n-2)}{\rm sym}(x_1^2 x_2^2 x_3^2).
\end{eqnarray*}
The following is a $D_6$-invariant harmonic polynomial of degree $6$, which 
is linearly independent of $f_6$:
\begin{eqnarray*}
f_{6,2}(v_k)=x_1 x_2 x_3 x_4 x_5 x_6.
\end{eqnarray*}

\noindent
\textbf{Substitute $v_k$ for $G$-invariant harmonic polynomials}\\
1. {\it Degree $4$.} \\ 
For $1 \leq k \leq n-2$, 
\begin{eqnarray*}
f_4(v_k)=\frac{1}{k}\left(1-3 \frac{k-1}{n-1} \right).
\end{eqnarray*}
For $k = n-1,n$,
\begin{eqnarray*}
f_4(v_k)=-\frac{2}{n}.
\end{eqnarray*}
For $n=4$, 
\begin{eqnarray*}
f_{4,2}(v_1)=0,\qquad f_{4,2}(v_2)=0, \qquad 
f_{4,2}(v_3)=-\frac{1}{16},\qquad f_{4,2}(v_4)=\frac{1}{16}.
\end{eqnarray*}
\noindent
2. {\it Degree $5$.} \\
For $n=5$,
\begin{eqnarray*}
f_5(v_1)=0, \qquad f_5(v_2)=0, \qquad f_5(v_3)=0, \qquad 
f_5(v_4)=-\frac{1}{25 \sqrt{5}}, \qquad f_5(v_1)=\frac{1}{25 \sqrt{5}}.
\end{eqnarray*}

\noindent
3. {\it Degree $6$.} \\ 
For $1 \leq k \leq n-2$,  
\begin{eqnarray*}
f_6(v_k)=\frac{1}{k^2}\left(1-15 \frac{k-1}{n-1}+30 
\frac{(k-1)(k-2)}{(n-1)(n-2)} \right).
\end{eqnarray*}
For $k=n-1,n$, 
\begin{eqnarray*}
f_6(v_k)=\frac{16}{n^2}.
\end{eqnarray*}
For $n=6$, 
\begin{eqnarray*} 
&f_{6,2}(v_1)=0, \qquad f_{6,2}(v_2)=0, \qquad f_{6,2}(v_3)=0, \qquad 
f_{6,2}(v_4)=0,& \\
& f_{6,2}(v_5)=-\frac{1}{216}, \qquad f_{6,2}(v_6)=\frac{1}{216}.&
\end{eqnarray*}

\begin{thrm} \label{thm:D_n}
 There is no choice of $J$, $R$ and $w$ for which $(\mathcal{X}(D_n,J),w)$ 
is a Euclidean $8$-design. 
\end{thrm}
\begin{proof}
The following is a $D_n$-invariant harmonic polynomial of degree $8$:
\begin{eqnarray*}
f_8={\rm sym}(x_1^8)-\frac{28}{n-1}{\rm sym}(x_1^2x_2^6)+\frac{70}{n-1}{\rm 
sym}(x_1^4 x_2^4).
\end{eqnarray*}
For $1\leq k \leq n-2$, 
\begin{eqnarray*}
f_8(v_k)=\frac{1}{k^3}\left( 1+7 \frac{k-1}{n-1} \right),
\end{eqnarray*}
and for $k=n-1,n$, 
\begin{eqnarray*}
f_8(v_k)=\frac{8}{n^3}. 
\end{eqnarray*}
Therefore $f_8(v_k)>0$ for all $k$.  
\end{proof}

\begin{thrm}
Assume $J$ contains $n$ or $n-1$.  Then
$\mathcal{X}(D_n,J)$ is not a 
tight Euclidean design
except for the sets in Table 3.  
\end{thrm}
\begin{proof}
By Theorem \ref{thm:D_n} and $D_n$-invariant harmonic polynomials, 
a proof is similar to that of Theorem \ref{thm:A_n}.
\end{proof}

\begin{table}
\begin{center}
\begin{tabular}{|c|c|c|c|c|c|}
\hline
$|R|$& $t$ & $n$ & $J$     & $r_i$ & $w_i$ \\
\hline
$1$  & $7$ & $8$ &$\{2,7\}$  & $r_2=r_7=1$ & $w_2=w_7=1$ \\
	 & $7$ & $8$ & $\{2,8\}$ & $r_2=r_8=1$ & $w_2=w_8=1$ \\
	
\hline
$2$  & $5$ & $6$ &$\{1,5\}$& $r_1=1$, $r_5 \ne 1$& $w_1=1$, 
$w_5=\frac{9}{8r_2^4}$\\
	 & $5$ & $6$ &$\{1,6\}$& $r_1=1$, $r_6 \ne 1$&$w_1=1$, 
$w_6=\frac{9}{8r_2^4}$\\
	 & $7$ & $4$ &$\{1,2,3,4\}$&$r_1=r_3=r_4=1$, $r_2\ne 1$ &$w_1=w_3=w_4=1$, 
$w_2=\frac{1}{r_2^6}$\\
\hline
\end{tabular}
\caption{Tight Euclidean $t$-designs from $\mathcal{X}(D_n,J)$, where 
$n\text{ or } n-1 \in J$ }
\end{center}
\end{table}
\begin{remark}
The tight Euclidean designs in Tables 1,2,3 are already known in 
\cite{Bajnok, Bannai-Bannai, Bannai-Bannai-Hirao-Sawa,
Etsuko, Etsuko2, Delsarte-Goethals-Seidel} 
\end{remark}

\begin{remark}
For each $G=F_4$, $H_3$, $H_4$, $E_6$, $E_7$, $E_8$,
by checking the cardinality of a union of several $v_k^G$,
we can prove $\mathcal{X}(G,J)$ is not a tight Euclidean design
except for known tight spherical designs \cite{Nozaki, Sali}. 
\end{remark}

\section{Concluding remarks}

In this paper we found some observations on
invariant cubature formulas and Euclidean designs
in connection with the Sobolev theorem.
First, we gave an alternative proof of
celebrated theorems by Xu on necessary and sufficient conditions
for the existence of cubature formulas with radial symmetry.
The new proof is much shorter and simpler
compared to the original one by Xu.
Thus
researchers in analysis will realize again the importance of the Sobolev theorem.
Moreover our proof gives a general interpretation
of the analytically-written conditions of Xu's theorems, and so
will promise researchers in algebra and combinatorics to be
more familiar with Xu's theorems.
Second
we extended the Neumaier-Seidel theorem to invariant Euclidean designs,
and thereby
classified tight Euclidean designs obtained from unions of the orbits of the corner vectors. 
The classification generalizes Bajnok's theorem
to other finite reflection groups beside
groups of type $B$.
Bajnok's theorem and results obtained in Section 4 may imply that
invariant cubature formulas of high degree could hardly exist.
Xu~\cite{Xu1} pointed out, however, that
the general Lie groups has been used for studying
cubature formulas in a different setting -- cubature rules on the fundamental
domain of the group, which are for exponential or trigonometric
functions -- and they yield Gaussian type cubature for algebraic
polynomials of very high orders; for instance see~\cite{Lie-Xu},~\cite{Patera-Moody}
for details.
We believe
this direction of research in analysis will also
motivate the study of cubature formulas
in other areas of mathematics.

\proof[Acknowledgements]
The authors started writing this paper during
their visit at the University of Texas at Brownsville, 2010,
under the sponsorship of the Japan Society for the Promotion of Science.
They would like to thank Oleg Musin for his hospitality.
The authors would also like to thank Akihiro Munemasa, Eiichi Bannai for
valuable comments to this work.
The second author would like to express his sincerest appreciation to Yuan Xu
for fruitful discussion about the content of Section 3 (\cite{Xu1}).


\begin{thebibliography}{10}

\bibitem{Bajnok1}
B. Bajnok,
On Euclidean designs.
{\em Adv. Geom.} $\mathbf{6}$ (2006), 423--438. 

\bibitem{Bajnok}
B.\ Bajnok. 
Orbits of the hyperoctahedral group as Euclidean designs. 
{\it J.\ Algebraic Combin.} 25 (2007), 375--397. 

\bibitem{Eiichi}
Ei.\ Bannai. Private communication.

\bibitem{Bannai-Bannai}
Ei.\ Bannai, Et.\ Bannai.
On Euclidean tight $4$-designs.
{\it J.\ Math.\ Soc.\ Japan} 58 (2006), 775--804. 

\bibitem{Bannai-Bannai2}
Ei.\ Bannai, Et.\ Bannai.
A survey on spherical designs and algebraic combinatorics on spheres.
{\it
Europ. J. Combin.} 30 (2009), 1392--1425.

\bibitem{Bannai-Bannai-Hirao-Sawa}
Ei.\ Bannai, Et.\ Bannai, M. Hirao, M. Sawa.
Cubature formulas in numerical analysis and Euclidean tight designs.
{\it Europ.\ J.\ Combin.} 31 (2010), 423--441.

\bibitem{Etsuko}
Et.\ Bannai.
New examples of Euclidean tight $4$-designs.
{\it Europ.\ J.\ Combin.} 30 (2009), 655--667. 


\bibitem{Etsuko2}
Et.\ Bannai.
On antipodal Euclidean tight $(2e+1)$-designs.  
{\it J.\ Algebraic Combin.} 24  (2006), 391--414. 

 \bibitem{Bourbaki}
N.\ Bourbaki.
{\it Lie Groups and Lie Algebras}: Chapters 4-6 (Elements of 
Mathematics). Springer, 2002).

\bibitem{Delsarte-Goethals-Seidel}
  P.\ Delsarte, J.M.\ Goethals, J.J.\ Seidel.
Spherical Codes and Designs.
{\it Geom.\ Dedicata} 6 (1977), 363-388. 
  
\bibitem{Delsarte-Seidel}
P.\ Delsarte, J.J.\ Seidel.
Fisher type inequalities for Euclidean $t$-designs.
{\it Lin.\ Algebra Appl.} 114--115 (1989), 213--230.

\bibitem{Dunkl-Xu}
C.F.\ Dunkl, Y.\ Xu.
{\it Orthogonal Polynomials of Several Variables}.
Cambridge University Press, 2001.

\bibitem{Erdelyi}
A.\ Erd$\check{\text{e}}$lyi et al.
{\it Higher Transcendental Functions II}.
(Bateman Manuscript Project),
MacGraw-Hill, 1953.

\bibitem{Goethals-Seidel}
J.M. Goethals, J.J. Seidel.
Cubature formulae, polytopes, and spherical designs.
The geometric vein, pp. 203--218, Springer, New York-Berlin, 1981.

\bibitem{Hirao-Sawa}
M.\ Hirao, M.\ Sawa.
On minimal cubature formulae of small degree for spherically symmetric integrals.
{\it SIAM J. Numer. Anal.} 47 (2009), 3195--3211. 

\bibitem{Lie-Xu}
H.\ Li, Y.\ Xu.
Discrete Fourier analysis on fundamental domain of $A_d$-lattice and on simplex in $d$-variables.
{\it J. Fourier Anal. Appl.} 16 (2010), 383 - 433.

\bibitem{Moller2}
H.M. M\"oller.
Lower bounds for the number of nodes in cubature formulae, Numerische
Integration (Tagung, Math.\ Forschungsinst.,
Oberwolfach, 1978). 221--230, {\it Internat.\ Ser.\ Numer.\ Math.\ }
45, Birkh$\ddot{{\rm a}}$user, Basel-Boston, Mass., 1979.

\bibitem{Mysovskikh2}
I.P. Mysovskikh.
Construction of cubature formulae (in Russian).
{\it Vopr. Vychisl. i Prikl. Mat. Tashkent} 32 (1975), 85--98.

\bibitem{Mysovskikh}
I.P. Mysovskikh.
Interpolatory Type Cubature formula (in Russian).
Nauka, Moscow, 1981.

\bibitem{Neumaier-Seidel} 
A. Neumaier, J. J. Seidel.
Discrete measures for spherical designs, eutactic stars and lattices. 
{\it Nederl. Akad. Wetensch. Indag. Math.} 50 (1988), 321--334.   

\bibitem{Nozaki}
H.\ Nozaki.
On the rigidity of spherical $t$-designs that are orbits of reflection 
groups $E_8$ and $H_4$.
{\it Europ. J. Combin.} 29 (2008), 1696--1703.   

\bibitem{Patera-Moody}
J.\ Patera, R.\ Moody.
Cubature formulae for orthogonal polynomials in terms of elements of finite order of compact simple Lie groups.
arXiv:1005.2773.

\bibitem{Sali}
A. Sali.
On the rigidity of spherical $t$-designs that are orbits of finite
reflection groups.
{\it Des.\ Codes Cryptogr.} 4 (1994), 157--170.

\bibitem{Salikhov}
G.N. Salikhov.
Cubature formulas for the hypersphere invariant under the $600$-hedral group.
{\it Dokl. Akad. Nauk SSSR} 223 (1975), 1075--1078.

\bibitem{Sobolev}
S.L. Sobolev.
Cubature formulas on the sphere which are invariant
under transformations of finite rotation groups (in Russian).
{\it Dokl. Akad. Nauk SSSR} 146 (1962), 310--313. 

\bibitem{Stroud}
A.H. Stroud.
{\it Approximate Calculation of Multiple Integrals}.
Prentice-Hall, Inc., Englewood Cliffs, N.J., 1971. xiii+431 pp.

\bibitem{Verlinden-Cools}
P.\ Verlinden, R.\ Cools.
On cubature formulae of degree $4k+1$ attaining 
M\"oller's lower bound for integrals with circular symmetry.
{\it Numer. Math.} $\mathbf{61}$ (1992), 395--407.

\bibitem{Xu}
Y. Xu.
Minimal cubature formulae
for a family of radial weight functions.
{\it Adv. Comput. Math.} 8 (1998), 367--380.

\bibitem{Xu1}
Y, Xu. Private communication.

\end{thebibliography}
\end{document}